\documentclass[reqno,a4paper, 11pt]{amsart}

\usepackage[utf8]{inputenc}

\usepackage[a4paper=true,pdfpagelabels]{hyperref}
\usepackage{graphicx}

\usepackage{amsfonts,epsfig}
\usepackage{latexsym}
\usepackage{mathabx}
\usepackage{amsmath}
\usepackage{amssymb}

\usepackage{color}

\newtheorem{theorem}{Theorem}
\newtheorem{lemma}[theorem]{Lemma}

\newtheorem{proposition}[theorem]{Proposition}

\newtheorem{lettertheorem}{Theorem}
\newtheorem{letterlemma}[lettertheorem]{Lemma}

\theoremstyle{definition}

\theoremstyle{remark}

\numberwithin{equation}{section}

\newcommand{\nm}[1]{\lVert#1\rVert}

\newcommand{\D}{\mathbb{D}}
\newcommand{\DD}{\widehat{\mathcal{D}}}
\newcommand{\Dd}{\widecheck{\mathcal{D}}}
\newcommand{\M}{\mathcal{M}}
\newcommand{\DDD}{\mathcal{D}}

\newcommand{\N}{\mathbb{N}}

\newcommand{\RR}{\mathbb{R}}

\newcommand{\C}{\mathbb{C}}

\renewcommand{\phi}{\varphi}

\def\a{\alpha}       \def\b{\beta}        \def\g{\gamma}
           
     \def\om{\omega}      
       \def\t{\theta}       
         \def\r{\rho}         
                  
\def\G{\Gamma}

\def\omg{\widehat{\omega}}
\def\nug{\widehat{\nu}}
\def\mug{\widehat{\mu}}

\def\etag{\widehat{\eta}}

\DeclareMathOperator{\supp}{supp}
\renewcommand{\H}{\mathcal{H}}

\newenvironment{Prf}{\noindent{\emph{Proof of}}}
{\hfill$\Box$ }

\begin{document}

\title[Littlewood-Paley inequalities for fractional derivative ]
{Littlewood-Paley inequalities for fractional derivative on Bergman spaces}


\author{Jos\'e \'Angel Pel\'aez}
\address{Departamento de An\'alisis Matem\'atico, Universidad de M\'alaga, Campus de
Teatinos, 29071 M\'alaga, Spain} \email{japelaez@uma.es}

\author{Elena de la Rosa}
\address{Departamento de An\'alisis Matem\'atico, Universidad de M\'alaga, Campus de
Teatinos, 29071 M\'alaga, Spain} 
\email{elena.rosa@uma.es}

\thanks{This research was supported in part by Ministerio de Econom\'{\i}a y Competitividad, Spain, projects
PGC2018-096166-B-100; La Junta de Andaluc{\'i}a,
projects FQM-210  and UMA18-FEDERJA-002.}

\subjclass[26A33, 30H20]{26A33,  30H20}
%\begin{abstract}
%\end{abstract}

\maketitle

%\tableofcontents \thispagestyle{empty}

\begin{abstract}

For any  pair $(n,p)$, $n\in\N$ and $0<p<\infty$, it has been recently proved in \cite{PR19} that
a radial weight $\omega$ on the unit disc of the complex plane $\mathbb{D}$ satisfies the Littlewood-Paley equivalence 
$$
\int_{\mathbb{D}}|f(z)|^p\,\omega(z)\,dA(z)\asymp\int_\mathbb{D}|f^{(n)}(z)|^p(1-|z|)^{np}\omega(z)\,dA(z)+\sum_{j=0}^{n-1}|f^{(j)}(0)|^p,$$
 for any analytic function $f$ in $\mathbb{D}$, if and only if $\omega\in\mathcal{D}=\widehat{\mathcal{D}} \cap \widecheck{\mathcal{D}}$.
A radial weight  
$\omega$ belongs to the class $\widehat{\mathcal{D}}$ if
  $\sup_{0\le r<1} \frac{\int_r^1 \omega(s)\,ds}{\int_{\frac{1+r}{2}}^1\omega(s)\,ds}<\infty$, and  $\omega \in \widecheck{\mathcal{D}}$   if there exists $k>1$ such that 
   $\inf_{0\le r<1} \frac{\int_{r}^1\omega(s)\,ds}{\int_{1-\frac{1-r}{k}}^1 \omega(s)\,ds}>1$.

In this paper we extend this result to the setting of fractional derivatives. Being precise,  
for an analytic function   $f(z)=\sum_{n=0}^\infty \widehat{f}(n) z^n$
 we consider the fractional derivative 
 $
  D^{\mu}(f)(z)=\sum\limits_{n=0}^{\infty} \frac{\widehat{f}(n)}{\mu_{2n+1}} z^n
 $
induced by a radial weight $\mu \in \mathcal{D}$
  where $\mu_{2n+1}=\int_0^1 r^{2n+1}\mu(r)\,dr$. Then, we prove that
for any $p\in (0,\infty)$, the Littlewood-Paley equivalence 
$$\int_\mathbb{D} |f(z)|^p \omega(z)\,dA(z)\asymp \int_\mathbb{D}|D^{\mu}(f)(z)|^p\left[\int_{|z|}^1\mu(s)\,ds\right]^p\omega(z)\,dA(z)$$
holds for any analytic function $f$ in $\mathbb{D}$ if and only if $\omega\in\mathcal{D}$.

 We also prove that for any  $p\in (0,\infty)$, the inequality
  $$\int_\mathbb{D}|D^{\mu}(f)(z)|^p\left[\int_{|z|}^1\mu(s)\,ds\right]^p\om(z)\,dA(z)
  \lesssim \int_\mathbb{D} |f(z)|^p \om(z)\,dA(z) $$
holds for any analytic function $f$ in $\D$ if and only if $\omega\in\DD$.

  \end{abstract}

\maketitle

\section{Introduction}

Let $\H(\D)$ denote the space of analytic functions in the unit disc $\D=\{z\in\C:|z|<1\}$.
For $f\in\H(\D)$ and $0<r<1$, set
    \begin{equation*}
    \begin{split}
    M_p(r,f)&=\left(\frac{1}{2\pi}\int_{0}^{2\pi} |f(re^{it})|^p\,dt\right)^{\frac1p},\quad
    0<p<\infty,
    \end{split}
    \end{equation*}
and $M_\infty(r,f)=\max_{|z|=r}|f(z)|$. For $0<p\le\infty$, the Hardy space $H^p$ consists of $f\in \H(\mathbb D)$ such that $\|f\|_{H^p}=\sup_{0<r<1}M_p(r,f)<\infty$. For a nonnegative function $\om\in L^1([0,1))$, the extension to $\D$, defined by 
$\om(z)=\om(|z|)$ for all $z\in\D$, is called a radial weight.
 For $0<p<\infty$ and such an $\omega$, the Lebesgue space $L^p_\om$ consists of complex-valued measurable functions $f$ on $\D$ such that
    $$
    \|f\|_{L^p_\omega}^p=\int_\D|f(z)|^p\omega(z)\,dA(z)<\infty,
    $$
where $dA(z)=\frac{dx\,dy}{\pi}$ is the normalized Lebesgue area measure on $\D$. The corresponding weighted Bergman space is $A^p_\om=L^p_\omega\cap\H(\D)$. Throughout this paper we assume $\widehat{\om}(z)=\int_{|z|}^1\om(s)\,ds>0$ for all $z\in\D$, for otherwise $A^p_\om=\H(\D)$.

\medskip  A well-known formula ensures that for each $n\in\N$ and $0<p<\infty$
\begin{equation}\label{LP1intro}
	\|f\|_{A^p_\om}^p\asymp\int_\D|f^{(n)}(z)|^p(1-|z|)^{np}\om(z)\,dA(z)+\sum_{j=0}^{n-1}|f^{(j)}(0)|^p,\quad f\in\H(\D),
	\end{equation}
	if $\omega$ is a standard radial weight, that is, $\om(z)=(\alpha+1)(1-|z|^2)^\alpha$ for some $-1<\alpha<\infty$. 
 Generalizations of this Littlewood-Paley formula  have been obtained in \cite{AS,PavP,Si}
 for different classes of radial weights. 
However, the question for which radial weights the above  equivalence \eqref{LP1intro} is valid has been a known open problem for decades.
This question has been recently solved in \cite[Theorem~5]{PR19}, in fact
 \eqref{LP1intro} holds for a radial weight $\om$ if and only if $\om\in\DDD=\DD\cap\Dd$. Recall that a radial weight $\om$ belongs to~$\DD$ if there exists a constant $C=C(\om)>1$ such that  $\widehat{\om}(r)\le C\widehat{\om}(\frac{1+r}{2})$ for all $0\le r<1$. Further, a radial weight $\omega$ belongs to~$\Dd$ if there exist constants $k=k(\om)>1$ and $C=C(\om)>1$ such that $\widehat{\om}(r)\ge C\widehat{\om}(1-\frac{1-r}{k})$ for all $0\le r<1$.

It is also worth mentioning that the inequality 
	\begin{equation}\label{LP2intro}
	\int_\D|f^{(n)}(z)|^p(1-|z|)^{np}\om(z)\,dA(z)+\sum_{j=0}^{n-1}|f^{(j)}(0)|^p\lesssim \|f\|_{A^p_\om}^p,\quad f\in\H(\D),
	\end{equation}
	holds for a radial weight $\om$ and each  pair $(n,p)$,  $n\in\N$ and $0<p<\infty$,  if and only if $\om\in\DD$
	 \cite[Theorem~6]{PR19}.

Throughout the next few lines we offer a brief insight to the classes of weights $\DDD$, $\DD$ and $\Dd$. Each standard radial weight
  belongs to $\DDD$,  while $\Dd\setminus\DDD$ contains weights that tend to zero exponentially. The class of rapidly increasing weights, introduced in \cite{PelRat}, lies entirely within $\DD\setminus\DDD$, and a typical example of such a weight is $\omega(z)=(1-|z|^2)^{-1}\left(\log\frac{e}{1-|z|^2}\right)^{-\alpha}$, where $\alpha>1$. However we emphasize that the containment in $\DD$ or $\Dd$ does not require continuity neither positivity. In fact, weights in these classes may vanish on a relatively large part of each outer annulus $\{z:r\le|z|<1\}$ of $\D$. For basic properties of the aforementioned classes, concrete nontrivial examples and more, see \cite{PelSum14,PelRat,PR19} and the relevant references therein.

The theory of weighted Bergman spaces $A^p_\omega$ induced by   non-radial weights is at its early stages, and plenty of essential
properties have not been described yet.
 However there have been  developments towards different directions during the last decades \cite{AlCo,HuLvSc}.
As for Littlewood-Paley formulas for derivatives, we recall that  
\eqref{LP1intro} holds if $\omega$ is a Bekoll\'e-Bonami weight \cite{AlCo,APR2}, see also \cite{BWZ,PR22} for related results.

On the other hand,  \eqref{LP1intro} can be extended to the setting of fractional derivatives when $\omega$ 
is a standard weight. Indeed, for  $f(z)=\sum_{n=0}^\infty \widehat{f}(n)z^n \in \H(\D)$  and $\beta>0$, consider
 the operator
\begin{equation}\label{eq:deffracder2}
D^{\beta} (f)(z)=\frac{2}{\G(\beta+1)}\sum\limits_{n=1}^{\infty} \frac{\G (n+\beta+1)}{\G(n+1)} \widehat{f}(n) z^n,\quad z\in\D,
\end{equation}
which basically coincides with the fractional derivative of order $\b>0$  introduced by Hardy and Littlewood in \cite[p. 409]{HLMathZ}.
The differences between \eqref{eq:deffracder2} and \cite[(3.13)]{HLMathZ}  are in the multiplicative factor $\frac{2}{\G(\beta+1)}$ and the inessential factor $z^\beta$. A floklore result  states that
\begin{equation}
\label{LP2fracHL}
\| f\|^p_{A^p_\omega}\asymp \int_\D|D^{\beta} (f)(z)|^p(1-|z|)^{\beta p}\om(z)\,dA(z),\quad f\in\H(\D),
\end{equation}
for any $\beta,p>0$ and any standard radial weight $\omega$, see \cite[Theorem~A]{BlascoCan95} for the range $p\ge 1$. Moreover, 
Flett proved in \cite[Theorem~6]{Flett72} that
\eqref{LP2fracHL} remains true for any $\beta,p>0$ if
$D^{\beta} f$ is replaced by the multiplier transformation
$f^{[\beta]}(z)=\sum_{n=0}^\infty (n+1)^\beta\widehat{f}(n)z^n$, 
which may also be regarded as fractional derivative
of order $\beta>0$.

For  a radial weight $\mu\in \DD$, we define 
the  fractional derivative of $f$ induced by $\mu$
\begin{equation*}
    \label{Dmu}
    D^{\mu}(f)(z)=\sum\limits_{n=0}^{\infty} \frac{\widehat{f}(n)}{\mu_{2n+1}} z^n, \; z \in \D.
\end{equation*}
 Here $\mu_{2n+1}$ are the odd moments of $\mu$, and in general from now on we write $\mu_x=\int_0^1r^x\mu(r)\,dr$
 for $\mu$ a radial weight and  $x\ge0$.
It is clear that $D^{\mu}(f)$ is a polynomial if $f$ is a polynomial and $D^{\mu}(f) \in \H(\D)$ for each $f \in \H(\D)$, by Lemma \ref{le:welldef} below. See  \cite{PeralaAASF20,ZhuPacific94} for related definitions or reformulations of classical and generalized fractional derivative,
and  observe that 
$D^{\mu}=D^{\beta} $ if $\mu$ is the standard weight $\mu(z)=\beta(1-|z|^2)^{\beta-1}$, $\b >0$.

The primary purpose of this paper is twofold:  extending the
  Littlewood-Paley formulas \eqref{LP1intro} and \eqref{LP2intro} replacing the higher order derivative $f^{(n)}$ by
  the fractional derivative 
  $D^\mu (f)$ induced by $\mu\in\DDD$, and describing the radial weights  such that the arising formulas hold. 
With this aim, observe that $\widehat{\mu}(z)\asymp (1-|z|)^{\beta}$ when   $\mu(z)=\beta(1-|z|^2)^{\beta-1}$, $\b >0$, so
an appropriate interpretation of the  Littlewood-Paley estimate 
\eqref{LP2fracHL}  is
$$ \|f\|_{A^p_\om}^p\asymp\int_\D|D^{\mu}(f)(z)|^p\widehat{\mu}(z)^p\om(z)\,dA(z),\quad f\in\H(\D),$$
when $\mu$ is a standard weight.

Our main result shows that the discussion above regarding standard weights actually describes a general phenomenon rather than a particular case, and moreover  describes the radial weights such that the formula holds.

\begin{theorem}\label{th:L-P-D}
Let $\om$ be a radial weight, $0<p<\infty$ and $\mu\in\DDD$. Then
    \begin{equation}\label{Eq:L-P-D}
    \|f\|_{A^p_\om}^p\asymp\int_\D|D^{\mu}(f)(z)|^p\widehat{\mu}(z)^p\om(z)\,dA(z),\quad f\in\H(\D),
    \end{equation}
if and only if $\om\in\DDD$.
\end{theorem}

In particular, as a byproduct of Theorem~\ref{th:L-P-D}  we obtain a proof of the folklore result
$$\| f\|^p_{A^p_\alpha}\asymp \int_\D|D^{\beta} (f)(z)|^p(1-|z|)^{\beta p+\alpha}\,dA(z),\quad f\in\H(\D),$$
for any $\beta,p>0$, and $\alpha>-1$. Here and throughout the paper 
$A^p_\alpha$ stands for the classical weighted Bergman spaces
induced by the standard radial weight
$\omega(z)=(\alpha+1)(1-|z|^2)^\alpha$.

En route  to the proof of Theorem~\ref{th:L-P-D} 
we will establish the following result,
which generalizes \cite[Theorem~5]{PR19} to the setting of fractional derivatives
induced by radial doubling weights.

\begin{theorem}\label{theorem:L-P-D-hat}
Let $\om$ be a radial weight, $0<p<\infty$ and $\mu\in\DDD$. Then,  there exists a constant $C=C(\om,\mu,p)>0$ such that
    \begin{equation}\label{eq:desLP}
   \int_\D|D^{\mu}(f)(z)|^p\widehat{\mu}(z)^p\om(z)\,dA(z)\le C\|f\|_{A^p_\om}^p,\quad f\in\H(\D),
    \end{equation}
     if and only if $\om\in\DD$.
\end{theorem}

The proof of \eqref{eq:desLP} 
 of 
 is strongly based on  
  the following inequality between the integral means of order $p$ of $D^{\mu} f$ and $f$, 
 \begin{equation}\label{eq:imdu} 
 M_p(r, D^{\mu} f)\leq C \frac{M_p(r,f)}{\mug\left(\frac{r}{\r}\right)}, \; 0<r<\r<1, \;0<p<\infty.
 \end{equation}
 The inequality \eqref{eq:imdu} is proved in Proposition \ref{pr:integralmeans} below
and it 
  is a natural extension of 
  \cite[Lemma~3.1]{PavP}. The proof of this last result  employes the Cauchy formula for  $f'$,  Minkowski's inequality for the case 
  $p\ge 1$ and factorization results of $H^p$ functions
  when $0<p<1$.
   However, 
   the proof of \eqref{eq:imdu}
   is strongly  based on smooth properties of universal Ces\'aro basis of polynomials introduced by  Jevti\'c  and Pavlovi\'c  \cite{JevPac98}.

Reciprocally, the other implication in the proof of
Theorem \ref{theorem:L-P-D-hat} uses ideas  from
 \cite[Theorem 6]{PR19} and  some technicalities. In particular,
 the proof reveals that \eqref{eq:desLP} holds if and only if the inequality there holds for all monomials only.

As for the proof of Theorem~\ref{th:L-P-D} we show, in Theorem~\ref{th:inverseM} below, that the  inequality
\begin{equation}\label{eq:revdes}
   \|f\|_{A^p_\om}^p \leq C\int_\D|D^{\mu}(f)(z)|^p\widehat{\mu}(z)^p\om(z)\,dA(z),\quad f\in\H(\D),
    \end{equation}
implies   that  $\om \in \M$. Recall that $\om\in\M$ if there exist constants $C=C(\om)>1$ and $k=k(\om)>1$ such that $\om_{x}\ge C\om_{kx}$ for all $x\ge1$. It is known that $\Dd\subset \M$ \cite[Proof of Theorem~3]{PR19} but $\Dd\not\subset \M$
 \cite[Proposition~14]{PR19}. However, \cite[Theorem~3]{PR19} ensures that $\DDD=\DD\cap\Dd=\DD\cap\M$,
 so Theorem~\ref{th:inverseM} together with Theorem \ref{theorem:L-P-D-hat} yields that $\om\in\DDD$ when
 \eqref{Eq:L-P-D} holds.

Concerning the reverse implication in the proof of Theorem~\ref{th:L-P-D}, we   construct ad hoc norms for the  weighted 
Bergman spaces $A^p_{\om}$ and $A^p_{\om \mug^p}$,  in the spirit of the decomposition results
\cite[Theorem 7.5.8]{Pabook} and \cite[Theorem~4]{PelRathg}. These two results are valid for $1<p<\infty$
and their proofs employ the boundedness of the Riesz projection on $L^p(\partial\D)$ to deal with the $H^p$-norm of polynomials of the type
$\Delta_{n_1,n_2}z=\sum_{k=n_1}^{n_2}z^k$.
However we use 
 universal Ces\'aro basis of polynomials instead of the polynomyals 
$\Delta_{n_1,n_2}$, their smooth properties allow us to get equivalent norms
 for any $0<p<\infty$.
Being precise, we  prove that
  there is  $k>1$ and a 
  shared  universal Ces\'aro basis of polynomials $\{V_{n,k}\}_{n=0}^\infty$ such that
 $\|f\|_{A^p_\omega}\asymp \sum\limits_{n=0}^{\infty} \omega_{k^n}\Vert V_{n,k}\ast f\Vert_{H^p}^p$ and 
 $\|D^\mu(f)\|_{A^p_{\omega\mug^p}}\asymp \sum\limits_{n=0}^{\infty} ({\omega\mug^p})_{k^n}\Vert V_{n,k}\ast f\Vert_{H^p}^p$
 for any $p\in (0,\infty)$,
 where  $\ast$ denotes the convolution.

Finally, we introduce the following notation that has already been used above in the introduction. The letter $C=C(\cdot)$ will denote an absolute constant whose value depends on the parameters indicated
in the parenthesis, and may change from one occurrence to another.
We will use the notation $a\lesssim b$ if there exists a constant
$C=C(\cdot)>0$ such that $a\le Cb$, and $a\gtrsim b$ is understood
in an analogous manner. In particular, if $a\lesssim b$ and
$a\gtrsim b$, then we write $a\asymp b$ and say that $a$ and $b$ are comparable.

\section{Preliminary results}

\subsection{Radial weights}
In this section we provide several characterizations of the classes of radial weight $\DD,\Dd$ and $\mathcal{M}$,
which will be used in the proofs of the main results of this paper.

For each $\b>0$ and $\om$ a radial weight, let us denote $\om_{[\b]}(s)=(1-s)^\beta\om(s).$
The next result gathers descriptions of the class $\DD$. 
\begin{letterlemma}
\label{caract. pesos doblantes}
Let $\om$ be a radial weight. Then, the following statements are equivalent:
\begin{itemize}
    \item[(i)] $\om \in \DD$;
    \item[(ii)] There exist $C=C(\om)\geq 1$ and $\a_0=\a_0(\om)>0$ such that
    $$ \omg(s)\leq C \left(\frac{1-s}{1-t}\right)^{\a}\omg(t), \quad 0\leq s\leq t<1;$$
    for all $\a\geq \a_0$;
    \item[(iii)]  
    $$ \om_x=\int_0^1 s^x \om (s) ds\asymp \omg\left(1-\frac{1}{x}\right),\quad x \in [1,\infty);$$
    \item[(iv)] There exists $C(\om)>0$ such that $\om_n\le C \om_{2n}$, for any $n\in \N$;
    \item[(v)] There exist $C(\om)>0$ and $\eta(\om)>0$ such that 
    $$\om_x\le C\left( \frac{y}{x}\right)^{\eta}\om_y,\quad 0<x\le y<\infty \; ;$$
    \item[(vi)] For some (equivalently for each) $\b >0$ there exists a constant $C=C(\om, \b)>0$ such that $x^{\b}(\om_{[\b]})_x\leq C \om_x, \, 0< x < \infty$.
\end{itemize}
\end{letterlemma}
\begin{proof}
The equivalences (i)--(v) can be found in \cite[Lemma~2.1]{PelSum14} and
(i)$\Leftrightarrow$(vi) is proved in \cite[Theorem~6]{PR19}, 
 where it is provided  a direct proof of (vi)$\Rightarrow$(i), but (i)$\Rightarrow$(vi) 
 is obtained by using the Littlewood-Paley inequality \cite[(1.5)]{PR19}.
 So, here we  give a detailed direct proof of (i)$\Rightarrow$(vi)
 for the convenience of the reader and the sake of completeness.

Let $\b>0$, $\om \in \DD$ and $x>0$. 
Observe that
$$x^{\b}(\om_{[\b]})_x= x^{\b}\int_0^{1-\frac{1}{x}} r^x (1-r)^{\b} \om (r)dr + x^{\b}\int_{1-\frac{1}{x}}^1 r^x (1-r)^{\b} \om (r)dr =I+II. $$
By  Fubini's theorem and Lemma \ref{caract. pesos doblantes}~(ii)
\begin{align*}
 I & = x^{\b +1}\int_0^{1-\frac{1}{x}} (1-r)^{\b} \left( \int_0^r s^{x-1}ds \right) \om (r)dr 
\\ & = x^{\b+1}\int_0^{1-\frac{1}{x}} s^{x-1} \left( \int_s^{1-\frac{1}{x}} (1-r)^{\b} \om (r)dr \right) ds 
 \\& \leq  x^{\b+1}\int_0^{1-\frac{1}{x}} s^{x-1} (1-s)^{\b} \omg (s) ds
 \\ & \lesssim x^{\b+\a+1}\omg\left(1-\frac{1}{x}\right)\int_0^{1-\frac{1}{x}} s^{x-1}  (1-s)^{\b+\a} ds
\\ & \lesssim \omg\left(1-\frac{1}{x}\right)\le \om_x.
 % \\&\lesssim  x^{\b+\a+1}\omg\left(1-\frac{1}{x}\right)\int_{1-\frac{1}{x}}^1  (1-s)^{\b+\a} ds \asymp \om_x
\end{align*}
%where in the last inequality we have used Lemma \ref{caract. pesos doblantes}(iii).
 Moreover, it is easy to observe that $II\leq  \om_x$. This finishes the proof.
\end{proof}

We will also need the following characterizations of the class $\Dd$.
\begin{letterlemma}
\label{caract. D check}
Let $\om$ be a radial weight. The following statements are equivalent:
\begin{itemize}
\item[(i)] $\om	\in \Dd$;
\item[(ii)] There exist $C=C(\om)>0$ and $\b=\b(\om)>0$ such that
$$\omg(s)\leq C \left(\frac{1-s}{1-t}\right)^{\b}\omg(t), \quad 0\leq  t\leq s<1;$$
\item[(iii)]There exist $k=k(\om)>1 $ and $C=C(\om)>0 $ such that
\begin{equation}
\label{D chek y k}
\int_r^{1-\frac{1-r}{k}}\om (s)ds \geq C \omg(r), \; 0\leq r<1.
\end{equation}
\end{itemize}
\end{letterlemma}

\begin{proof}
The condition (iii) is just a reformulation of the definition of the class $\Dd$, so 
we omit the proof of (i)$\Leftrightarrow$(iii).
Next, assume that (i) holds and % Then there exist $k=k(\om)>1$ and $C=C(\om)>1$ such that $\omg(r)\geq C \omg \left(1-\frac{1-r}{k}\right).$  
consider the sequence $\{r_n\}_{n=0}^\infty=\{1-\frac{1}{k^n}\}_{n=0}^\infty$.
  If $0 \leq t\leq s <1$, there exist $m, \, n \in \N\cup\{0\}$, $m \geq n$ such that $r_n \leq t < r_{n+1}$ and $r_m \leq s < r_{m+1}$. 
  %By definition, $\omg(r_n)\geq C \omg(r_{n+1})$ for all $n \in \N$, and 
If $n+1\le m$, then  
\begin{align*}
 \omg(s)&\leq \omg(r_m)\leq \frac{1}{C} \omg (r_{m-1})\leq \dots\leq \frac{1}{C^{m-n-1}}\omg(r_{n+1})\leq \frac{C^2}{k^{(m-n+1)\log_k C}} \omg(t)
 \\ & \leq C^2 \left(\frac{1-s}{1-t}\right)^{\log_k C} \omg(t).
\end{align*}
%so (ii) holds with $\beta=\log_k C>0$, where $C$ is the constant in the definition of $\Dd$.
Next, if $m=n$, then for any constant $C_1\ge k^\beta$
$$\frac{\omg(s)}{\omg(t)}\le 1\le C_1\frac{1}{k^\beta}\le C_1\left( \frac{1-r_{n+1}}{1-r_n} \right)^\beta\le C_1
\left( \frac{1-s}{1-t} \right)^\beta,$$
so (ii) holds for any exponent $\beta\in ( 0, \log_k C]$.
 Reciprocally, assume (ii) and let be $k>1$. By (ii) there exist $C=C(\om )>0$ and $\b=\b(\om)>0$ such that $\frac{C}{k^{\b}}\omg(r)\geq \omg \left(1-\frac{1-r}{k}\right)$. So taking $k >C^{\frac1{\b}}$, (i) holds. This finishes the proof.

\end{proof}

A proof of the following description of the  weights $\om\in\M$, 
in terms of the moments of $\om$, can be found in \cite[Theorem~2]{PR19}.
\begin{letterlemma}
\label{caract. M}
Let $\om$ be a radial weight. The following statements are equivalent:
\begin{itemize}
\item[(i)] $\om 	\in \M$;
\item[(ii)] For each $\beta>0$, there is $C=C(\beta,\om)$ such that 
$$\om_x\le C x^\beta \left( \om_{[\beta]}\right)_x,\quad x\ge 1.$$
\end{itemize}
\end{letterlemma}

The next result is an essential tool to deal with the inequality 
$$ \|f\|_{A^p_\om}^p\lesssim \int_\D|D^{\mu}(f)(z)|^p\widehat{\mu}(z)^p\om(z)\,dA(z),\quad f\in\H(\D),$$
(see Theorem~\ref{th:inversesuf} below), so to
 prove Theorem~\ref{th:L-P-D}.

\begin{lemma}
\label{prop pesos}
Let $\om$ be a radial weight. Then, the following statements holds:

\begin{itemize}
\item[(i)] If $\om \in \Dd$ and $\varphi:[0,1)\to (0,\infty)$ is decreasing, then $\om\varphi \in \Dd$. Furthermore, if there exist $k=k(\om)>1$ and $C=C(\om)>0$ such that
\begin{equation}
\label{1}
C\int_{1-\frac{1-r}{k}}^{1}\om(s)\,ds\leq \int_r^{1-\frac{1-r}{k}}\om (s)ds , \; 0 \leq r <1
\end{equation}
then 
\begin{equation}
\label{2}
C \int_{1-\frac{1-r}{k}}^{1}\om(s) \varphi(s)\,ds\leq \int_r^{1-\frac{1-r}{k}}  \om(s)\varphi(s) \,ds, \; 0 \leq r <1.
\end{equation}
\item[(ii)] If $\om \in \DDD$, $\mu \in \DD$ and $0<p<\infty$, then $\om\mug^p \in \DDD$.
%\item[(iii)] If $\om\in \DDD$ and $\mu \in \DD$, then $\om\mug^p \in \DDD$.
\end{itemize}
\end{lemma}
\begin{proof}

By Lemma~\ref{caract. D check}~(iii), 
the proof of (i) follows from the proof of \eqref{1}$\Rightarrow$\eqref{2}. Indeed,
since $\varphi$ is decreasing
\begin{align*}
\int_r^{1-\frac{1-r}{k}} \om(s)\varphi(s) \,ds &\geq \varphi\left(1-\frac{1-r}{k} \right) \int_r^{1-\frac{1-r}{k}} \om(s) ds  
\\& \geq  C \varphi\left(1-\frac{1-r}{k} \right)\int_{1-\frac{1-r}{k}}^{1}\om(s)\,ds 
\\ & \geq C \int_{1-\frac{1-r}{k}}^{1}\om(s) \varphi(s)\,ds,
\end{align*}
and \eqref{2} holds.

Throughout the rest of the proof let us denote $\nu=\om\mug^p.$
Since $\om \in \DDD$, by Lemma~\ref{caract. D check}~(iii)  there exists $k=k(\om)>1$ such that
\begin{equation*}
\nug(r)\leq \mug(r)^p\omg(r)\lesssim \omg\left(\frac{1+r}{2}\right)\mug(r)^p\lesssim  \mug(r)^p\int_{\frac{1+r}{2}}^{1-\frac{1-r}{2k}} \om(s)ds.
\end{equation*}
Moreover, by Lemma~\ref{caract. pesos doblantes}~(ii)  there exist $C=C(\mu)>0$ and $\a=\a(\mu)>0$ such that $\mug(r)\leq 2^{\a}k^{\a}C \mug\left(1-\frac{1-r}{2k}\right)$, so
$$\nug(r)\lesssim \int_{\frac{1+r}{2}}^{1-\frac{1-r}{2k}} \om(s)\mug(s)^p ds\le
 \nug\left(\frac{1+r}{2}\right),$$
that is $\nu \in \DD$. Moreover, $\nu\in \Dd$ by (i). This finishes the proof.
%$(iii)$ follows from $(i)$ and $(ii)$.
\end{proof}

\subsection{Universal Ces\'aro basis of polynomials}

In this section, 
we  establish some notation and previous results on universal Ces\'aro basis of polynomials,
which will be strongly used in the proofs of  Theorem~\ref{th:L-P-D} and Theorem~\ref{theorem:L-P-D-hat}.

The Hadamard product of a polynomial  $W(z)=\sum\limits_{k \in J}b_k z^k$,  where $J$ denotes a finite subset of $\N$ 
and $f(z)=\sum_{k=0}^\infty a_k z^k\in\H(\D)$ is 
$$(W\ast f)(z)=\sum_{k=0}^\infty a_k b_kz^k,\quad z\in\D.$$ 
 Furthermore, it is easy to observe that
    \begin{equation*}
    (W\ast f)(e^{it})
    =\frac{1}{2\pi}\int_{-\pi}^\pi W(e^{i(t-\theta)})f(e^{i\t})\,d\t.
    \end{equation*}
For a given $C^\infty$-function $\Phi:\mathbb{R}\to\C$ with compact support,
set
    $$
    A_{\Phi,m}=\max_{x\in\mathbb{R}}|\Phi(x)|+m\max_{x\in\mathbb{R}}|\Phi^{(m)}(x)|,\quad m\in\N\cup\{0\},
    $$
and define the polynomials
    \begin{equation*}
    W_n^\Phi(z)=\sum_{k\in\mathbb
    Z}\Phi\left(\frac{k}{n}\right)z^{k},\quad n\in\N.
    \end{equation*}

The next result can be  found in \cite[pp.~111--113]{Pabook}.

\begin{lettertheorem}
\label{Th Polinomios Cesaro}
Let $\Phi:\RR\to\C$ be a compactly supported $C^{\infty}$-function. Then the following statements hold:
\begin{itemize}
\item[(i)] There exixts a constant $C>0$ such that
$$ \vert W_n^{\Phi}(e^{i\t})\vert\leq C \min \left\{n \max\limits_{s \in \RR}\vert \Phi (s)\vert, n^{1-m}\vert \t\vert^{-m}\max\limits_{s \in \RR}\vert \Phi^{(m)} (s)\vert\right\}$$
for all $m \in \N\cup\{0\}$, $n \in \N$ and $0<\vert \t\vert<\pi$.
\item[(ii)] If $0<p\leq 1$ and $m \in \N$ with $mp>1$, there exists a constant $C=C(p)>0$ such that
$$ \left(\sup\limits_{n}\vert (W_n^{\Phi}* f)(e^{i\t})\vert\right)^p\leq C A_{\Phi, m}^p M (\vert f \vert^p)(e^{i\t})$$
for all $f \in H^p$, where $M$ denotes the Hardy-Littlewood maximal-operator
$$ M(f)(e^{i\t})=\sup\limits_{0<h<\pi}\frac{1}{2h}\int_{\t-h}^{\t+h}\vert f (e^{it})\vert dt.$$
\item[(iii)] For each $0<p<\infty$ and $m \in \N$ with $mp>1$, there exists a constant $C=C(p)>0$ such that 
$$\Vert W_n^{\Phi}*f\Vert_{H^p}\leq C A_{\Phi,m}\Vert f \Vert_{H^p}$$
for all $f \in H^p$ and $n \in \N$.
\end{itemize}
\end{lettertheorem}

The property (iii) shows that the polynomials $\{W_n^\Phi\}_{n\in\N}$ can be seen as a universal C\'esaro basis for $H^p$ for any $0<p<\infty$. In the statement of the next result, we consider a 
 particular family of polynomials $\{W_n^\Phi\}_{n\in\N}$  which
 play a key role in this manuscript.

\begin{proposition}\label{pr:cesaro}
Let $k \in \N$, $k>1$ and $\Psi:\mathbb{R}\to\mathbb{R}$ be a $C^\infty$-function such that $\Psi\equiv 1$ on $(-\infty,1]$, $\Psi\equiv 0$ on $[k,\infty)$ and $\Psi$ is decreasing and positive on $(1,k)$. Set $\psi(t)=\Psi\left(\frac{t}{k}\right)-\Psi(t)$ for all $t\in\mathbb{R}$. Let $V_{0,k}(z)=\sum\limits_{j=0}^{k-1} \Psi(j) z^j$
and
    \begin{equation*}\label{vnk}
    V_{n,k}(z)=W^\psi_{k^{n-1}}(z)=\sum_{j=0}^\infty
    \psi\left(\frac{j}{k^{n-1}}\right)z^j=\sum_{j=k^{n-1}}^{k^{n+1}-1}
    \psi\left(\frac{j}{k^{n-1}}\right)z^j,\quad n\in\N.
    \end{equation*}
 
Then,
    \begin{equation}
    \label{propervn1}
    f(z)=\sum_{n=0}^\infty (V_{n,k}\ast f)(z),\quad z\in\D,\quad f\in\H(\D),
    \end{equation}
    and for each $0<p<\infty$ there exists a constant $C=C(p,\Psi,k)>0$ such that
    \begin{equation}
    \label{propervn2}
   \|V_{n,k}\ast f\|_{H^p}\le C\|f\|_{H^p},\quad f\in H^p, \quad n \in \N.
    \end{equation}
\end{proposition}

If $k=2$ we simply denote $V_{n,2}=V_n$. A
proof of Proposition~\ref{pr:cesaro} for this choice appears in
 \cite[p. 175--177]{JevPac98} or \cite[p. 143--144]{Pabook2}. 
 For the convenience of the reader and the sake of completeness, we present a proof of Proposition~\ref{pr:cesaro}  
   following the ideas in 
\cite{JevPac98,Pabook2}.

\begin{Prf}{\em{Proposition~\ref{pr:cesaro}}.}     
Let us denote by $\{\widehat{V_{n,k}}(j)\}_j$ the sequence of Taylor coefficients of $V_{n,k}$.
Since $\sum\limits_{j=0}^{\infty} |\widehat{f}(j)||z|^j$ converges for each $z \in \D$, 
$\supp \widehat{V_{n,k}}\subset \N\bigcap [k^{n-1},k^{n+1})$ and  $\vert \widehat{V_{n,k}}(j)\vert \leq 2$ for all $n \in \N$ and $j \in \N$, 
\begin{align*}
\left| \sum_{n=1}^\infty (V_{n,k}\ast f)(z) \right| \le 2 \sum_{n=1}^\infty \sum\limits_{j=k^{n-1}}^{k^{n+1}} |\widehat{f}(j)||z|^j \le 4\sum_{j=0}^\infty |\widehat{f}(j)||z|^j,
\end{align*}
that is,  $\sum_{n=0}^\infty (V_{n,k}\ast f)(z)$ converges for each $z \in \D$.
 Let us prove that $ \sum\limits_{n=0}^{\infty}\widehat{V_{n,k}}(j)=1$, for each $j=0,1,2,\dots .$
 
If $0\leq j \leq k-1$, $\sum\limits_{n=0}^{\infty}\widehat{V_{n,k}}(j)= \widehat{V_{0,k}}(j)+ \widehat{V_{1,k}}(j)=
\Psi\left(\frac{j}{k}\right)=1$. 
On the other hand,
 if $j \geq k$ then $\Psi(j)=0$ and $\sum\limits_{n=0}^{\infty}\widehat{V_{n,k}}(j)= \lim \limits_{m\to \infty} \sum\limits_{n=1}^{m}\psi\left(\frac{j}{k^{n-1}}\right)=\lim\limits_{m\to \infty} \Psi\left(\frac{j}{k^{m}}\right)=1$.

Therefore, it is clear that \eqref{propervn1} holds for polynomials. Let us show that \eqref{propervn1} holds for each $f \in \H(\D)$. 
Let be $S_nf(z)=\sum_{l=0}^n\widehat{f}(l)z^l$, the $n$-th partial sum of $f$.
Fixed $z \in \D$, 
\begin{align*}
\left|f(z)-\sum_{n=0}^\infty (V_{n,k}\ast f)(z)\right|
&\leq |f(z)-S_{k^{m}} f (z)|+\left|S_{k^{m}} f(z)- \sum_{n=0}^\infty (V_{n,k}\ast f)(z)\right|
\\ & =I(f,m,z)+II(f,m,z),
\end{align*}
where $I(f,m,z)=|f(z)-S_{k^{m}} f (z)|$ and $II(f,m,z)=\left|S_{k^{m}} f(z)- \sum_{n=0}^\infty (V_{n,k}\ast f)(z)\right|$.
We have that $\lim_{m\to \infty}I(f,m,z)=0$, and
 using
\eqref{propervn1}  for $S_{k^{m}}f$, 
\begin{equation*}\begin{split}
II(f,m,z)&= 
 \left| \sum_{n=0}^\infty(V_{n,k}\ast S_{k^{m}} f)(z)- \sum_{n=0}^\infty (V_{n,k}\ast f)(z)\right|\\
 &= \left| \sum_{n=0}^{m+1}(V_{n,k}\ast f)(z)- \sum_{n=0}^\infty (V_{n,k}\ast f)(z)\right|
 %%%\\&= \left| \sum_{n=m+2}^\infty(V_{n,k}\ast f)(z)\right|
\\ & \lesssim \sum_{j=k^{m+1}}^\infty |\widehat{f}(j)||z|^j,
\end{split}\end{equation*}
so $\lim_{m\to \infty}II(f,m,z)=0$. Consequently  \eqref{propervn1} holds  for  any $f \in \H(\D)$.

Finally, \eqref{propervn2} follows from Theorem \ref{Th Polinomios Cesaro}~(iii).
\end{Prf}
 
\section{Proof of Theorem~\ref{theorem:L-P-D-hat}.}

To begin with, we will prove some technical lemmas. The first one ensures that
the definition of $D^\mu$ makes sense when $\mu\in\DD$.

\begin{lemma}\label{le:welldef}
Let $\mu\in\DD$ and $f\in\H(\D)$. Then, the fractional derivative $D^\mu f\in \H(\D)$ .
\end{lemma}
\begin{proof}
By Lemma~\ref{caract. pesos doblantes}~(ii), 
  there exist $C=C(\mu)>0$ and $\a=\a(\mu)>0$ such that 
%$$\mug(s)\leq C_1\mug(t)\left(\frac{1-s}{1-t}\right)^{\a}, \; 0\leq s\leq t <1. $$
%Then 
$$\mu_{2k+1}\geq C \mug\left(1-\frac{1}{2(k+1)}\right)\geq \frac{C\mug(0)}{ 2^{\a}}\frac{1}{(k+1)^{\a}}=\frac{C}{(k+1)^{\a}}, \; k \in \N\cup\{0\}.$$
So, for each $k \in \N\cup\{0\}$,
$\sqrt[k]{\frac{\vert \widehat{f}(k)\vert}{\mu_{2k+1}}}\leq \sqrt[k]{\frac{(k+1)^{\a}}{C}}\sqrt[k]{\vert \widehat{f}(k)\vert}.$
 Then,
   it follows that
$\limsup\limits_{k\to\infty} \sqrt[k]{\frac{\vert \widehat{f}(k)\vert}{\mu_{2k+1}}}\leq 1$, and therefore $D^{\mu}(f)\in \H(\D).$
\end{proof}

\begin{lemma}
\label{suma}
Let $\mu \in \DD$,  $\gamma>0$ and $k \in \N\setminus\{1\}$. Then, 
\begin{equation}\label{eq:suma1}
 \sum\limits_{n=0}^{\infty}\frac{r^{k^{n}}}{\mu_{k^{n}}^{\gamma}}\asymp \int_0^r \frac{dt}{(1-t)\mug(t)^\gamma},\quad 0\le r<1.
 \end{equation}
Moreover, if $\mu\in\DDD$ 
\begin{equation}\label{eq:suma2}
 1+\sum\limits_{n=0}^{\infty}\frac{r^{k^{n}}}{\mu_{k^{n}}^{\gamma}}\asymp \frac{1}{\mug(r)^{\gamma}},\quad 0\le r<1.
 \end{equation}
\end{lemma}
\begin{proof}
Since $\mu\in\DD$,
\begin{equation*}\begin{split}
\sum\limits_{n=0}^{\infty}\frac{r^{k^{n}}}{\mu_{k^{n}}^{\gamma}} &\asymp  
\frac{r}{\mu_1^\gamma}+\sum\limits_{n=1}^{\infty}\frac{r^{k^{n}}}{\mu_{k^{n}}^{\gamma}}\frac{1}{k^n}\sum_{j=k^{n-1}}^{k^n-1}1
\\ & \lesssim \frac{r}{\mu_1^\gamma}+
\sum\limits_{n=1}^{\infty}\sum_{j=k^{n-1}}^{k^n-1}
\frac{r^j}{(j+1)\mu^\g_{j}}
\\ &
\asymp \sum_{j=1}^\infty \frac{r^j}{(j+1)\mu^\g_{2j+1}}, \quad 0\le r<1.
\end{split}\end{equation*}
Analogously, it can be proved that
$$\sum\limits_{n=0}^{\infty}\frac{r^{k^{n}}}{\mu_{k^{n}}^{\gamma}}\gtrsim \sum_{j=1}^\infty \frac{r^j}{(j+1)\mu^\g_{2j+1}}, \quad 0<r<1.$$

Now, arguing as in \cite[(2.9)]{PelRatAdv2016}, it follows
\begin{equation*}
\label{s}
\sum_{j=1}^\infty \frac{r^j}{(j+1)\mu^\g_{2j+1}} \asymp \int_0^r \frac{dt}{(1-t)\mug(t)^\gamma},  \quad 0<r<1, 
\end{equation*}
and  we get \eqref{eq:suma1}.

Next, bearing in mind Lemma~\ref{caract. D check}(ii),
\begin{equation}\label{eq:suma3}
1+\int_0^r \frac{dt}{(1-t)\mug(t)^\gamma}
\le 1+ C^{\gamma}\frac{(1-r)^{\b\gamma}}{\mug(r)^\gamma}\int_0^r \frac{dt}{(1-t)^{1+\b\gamma}}
\lesssim \frac{1}{\mug(r)^\gamma}, \quad 0\le r<1.
\end{equation}
On the other hand, it is clear that
\begin{equation}\label{eq:suma4}
1+\int_0^r \frac{dt}{(1-t)\mug(t)^\gamma}
\gtrsim  \frac{1}{\mug(r)^\gamma}, \quad 0\le r\le \frac{1}{2},
\end{equation}
and because $\mu\in\DD$,
\begin{equation}\label{eq:suma5}
1+\int_0^r \frac{dt}{(1-t)\mug(t)^\gamma}
\ge \int_{2r-1}^r \frac{dt}{(1-t)\mug(t)^\gamma} \gtrsim \frac{1}{\mug(2r-1)^\gamma}\gtrsim \frac{1}{\mug(r)^\gamma}, \quad \frac{1}{2}<r<1.
\end{equation}
Consequently, \eqref{eq:suma1} together with \eqref{eq:suma3}, \eqref{eq:suma4} and \eqref{eq:suma5} implies \eqref{eq:suma2}.
This finishes the proof.
\end{proof}

Now we will prove a generalization of \cite[Lemma 3.1]{PavP} to the setting of fractional derivatives
induced by doubling weights, which is interesting on its own right.
 
\begin{proposition}
\label{pr:integralmeans}
Let $0<p<\infty$ and $\mu\in\DDD$. Then, there exists a constant $C=C(p,\mu)>0$ such that
\begin{equation}
\label{eq:integralmeans}
     M_p(r, D^{\mu}(f))\le C\frac{M_p(\r,f)}{\mug\left(\frac{r}{\r}\right)},\quad 0\le r<\r<1, \quad f\in\H(\D).
\end{equation}     
\end{proposition}
\begin{proof}
We will split the proof in two cases according to the value of $r$ and $\r$.

{\bf{Case $\mathbf{1}$.  $\mathbf{\frac{1}{2}\le \frac{r}{\r}<1}$.}}
Bearing in mind \eqref{propervn1}, 
\begin{equation}\begin{split}\label{geq:p>1}
 M_p(r, D^{\mu}(f))
 \leq \sum\limits_{n=0}^{1} \|V_n* (D^{\mu} f )_r\|_{H^p} + \sum\limits_{n=2}^{\infty} \|V_n* (D^{\mu} f )_r\|_{H^p},  
\quad f \in \H(\D),
\end{split}\end{equation}
 for all $1<p<\infty$ and 
\begin{equation}\begin{split}\label{geq:p<1}
M_p^p(r, D^{\mu}(f)) \leq  \sum\limits_{n=0}^{1} \|V_n* (D^{\mu} f )_r\|^p_{H^p} + \sum\limits_{n=2}^{\infty} \|V_n* (D^{\mu} f )_r\|_{H^p}^p, \quad f\in \H(\D),
\end{split}\end{equation}
for $0<p\leq 1$, where $V_n=V_{n,2}$ are the polynomials defined in the statement of Proposition~\ref{pr:cesaro}.

Firstly,  
\cite[Lemma~3.1]{PavP} yields 
\begin{equation}\begin{split}\label{geq:im1}
\|V_0* (D^{\mu} f )_r\|_{H^p}& \leq \frac{|\widehat{f}(0)|}{\mu_1}+r\frac{|\widehat{f}(1)|}{\mu_3}
\\ & \leq
\frac{M_p(r, f)}{\mu_1}+r\frac{M_p(\frac{\r}{2}, f')}{\mu_3} 
\\ & \le \frac{M_p(r, f)}{\mu_1} + C(p)2r 
\frac{M_p(\r, f)}{\r\mu_3} 
\le  C(p) \frac{M_p(\r, f)}{\mug(\frac{r}{\r})}, \quad\frac{1}{2}\le \frac{r}{\r}<1.
\end{split}\end{equation}
The inequality 
\begin{equation}\label{geq:im1n}
\|V_1* (D^{\mu} f )_r\|_{H^p}\le C(p) \frac{M_p(\r, f)}{\mug(\frac{r}{\r})}, \quad \frac{1}{2}\le \frac{r}{\r}<1,
\end{equation}
can be proved analogously.
 Next, we will estimate from above the series in \eqref{geq:p>1} and \eqref{geq:p<1}. 
For each  $n \in \N, \, n\geq 2$, 
 let us consider 
the
 function
$$\varphi_n(x)=\frac{\left(\frac{r}{\r}\right)^x}{\mu_{2x+1}}\chi_{[2^{n-1}, 2^{n+1}-1]}(x), \quad \frac{1}{2}\le \frac{r}{\r}<1,$$
and fix $m \in \N$  such that $mp>1$. 
Observe that for each $k\in\N$, there is $C=C(k)>0$ such that
\begin{equation}
\label{gprop.}
 \int_0^1 s^{x}\left(\log\frac{1}{s}\right)^k\mu(s)ds\leq C\mu_{x-1}\le C \mu_x, \quad\text{for any}\,\,  x\ge 2.
\end{equation}
Since
$$ \varphi_n '(x)=-\frac{2\int_0^1 s^{2x+1}\log\frac{1}{s}\mu(s)ds}{(\mu_{2x+1})^2}\left(\frac{r}{\r}\right)^x +\frac{\left(\frac{r}{\r}\right)^x}{\mu_{2x+1}}\log \frac{r}{\r},\quad x\in (2^{n-1}, 2^{n+1}-1), $$ 
using \eqref{gprop.} and the inequality $\frac{1}{2}\le \frac{r}{\r}$, it follows that there is an absolute constant $C>0$ such that
$$\vert \varphi_n'(x)\vert \le C \frac{\left(\frac{r}{\r}\right)^x}{\mu_{2x+1}}, \quad x\in (2^{n-1}, 2^{n+1}-1),
 \quad\frac{1}{2}\le \frac{r}{\r}<1.$$ 

By this way, using \eqref{gprop.} and an induction process on $m$,  
there is $C=C(m)=C(p)>0$ such that  
 $$\vert \varphi_n ^{(m)}(x)\vert \le C \frac{\left(\frac{r}{\r}\right)^x}{\mu_{2x+1}}, \quad x\in (2^{n-1}, 2^{n+1}-1), 
 \quad \frac{1}{2}\le \frac{r}{\r}<1.$$
 
Now, for each $n\in\N\setminus\{1\}$, choose a $C^\infty$-function $\Phi_n$ with compact support contained in $[2^{n-2},2^{n+2}]$ such that $\Phi_n=\varphi_n$ in $[2^{n-1}, 2^{n+1}-1]$. Then, bearing in mind Lemma \ref{caract. pesos doblantes}~(v), 
there is $C=C(p,\mu)>0$ such that
\begin{equation}\label{geq:im2}
A_{\Phi_n,m}\le C \max\limits_{x \in [2^{n-1}, 2^{n+1}-1]}\frac{\left(\frac{r}{\r}\right)^x}{\mu_{2x+1}}\leq C \frac{\left(\frac{r}{\r}\right)^{2^{n-1}}}{\mu_{2^{n+2}+1}} %%%%\%%%\frac{r^{2^{n-2}}}{\mu_{2^{n}}}
\le C \frac{\left(\frac{r}{\r}\right)^{2^{n-1}}}{\mu_{2^{n}}},\quad \frac{1}{2}\le \frac{r}{\r}<1.
\end{equation}
Moreover,
\begin{equation*}
\begin{split}
 V_n*(D^{\mu} f)_r(z)&=\sum\limits_{k=2^{n-1}}^{2^{n+1}-1}\psi\left(\frac{k}{2^{n-1}}\right)\frac{\left(\frac{r}{\r}\right)^k}{\mu_{2k+1}}\widehat{f}(k)\r^k z^k
 \\ &=\sum\limits_{k=2^{n-1}}^{2^{n+1}-1}\psi\left(\frac{k}{2^{n-1}}\right)\Phi_n(k)\widehat{f}(k)\r^{k}z^k\\
 &=(W_1^{\Phi_n}*V_n*f_{\r})(z).
\end{split}
\end{equation*}
So, Theorem \ref{Th Polinomios Cesaro}(iii), \eqref{geq:im2} and \eqref{propervn2} imply that
for each $n\in\N\setminus\{1\}$
\begin{equation*}
\begin{split}
\Vert V_n*(D^{\mu} f)_r\Vert_{H^p} &\le C A_{\Phi_n, m}\Vert V_n*f_{\r}\Vert_{H^p}
\\ & \le C \frac{\left(\frac{r}{\r}\right)^{2^{n-1}}}{\mu_{2^{n}}}\Vert V_n*f_{\r}\Vert_{H^p}
\le C \frac{\left(\frac{r}{\r}\right)^{2^{n-1}}}{\mu_{2^{n}}}M_p(\r, f), \quad\frac{1}{2}\le \frac{r}{\r}<1,
\end{split}
\end{equation*}
where $C=C(p,\mu)>0$.
So, by Lemma \ref{suma}, there is $C=C(\mu,p)>0$ such that 
\begin{equation}
\begin{split}\label{geq:mi3}
\sum\limits_{n=2}^{\infty} \|V_n* (D^{\mu} f )_r\|_{H^p}
  &\le C M_p(\r, f)\left(\sum\limits_{n=2}^{\infty}\frac{\left(\frac{r}{\r}\right)^{2^{n-1}}}{\mu_{2^{n}}}\right)
\\ & \le C  \frac{M_p(\r,f)}{\mug\left(\frac{r}{\r}\right)}, \quad  \frac{1}{2}\le \frac{r}{\r}<1, 
\, \, f\in \H(\D),\, \,1<p<\infty.
\end{split}
\end{equation}
Analogously, using Lemma \ref{suma} again 
%and if $0< p \leq 1$
\begin{equation}
\begin{split}\label{geq:mi4}
\sum\limits_{n=2}^{\infty} \|V_n* (D^{\mu} f )_r\|_{H^p}^p  &
\le C M_p^p(\r, f)\left(\sum\limits_{n=2}^{\infty}\frac{\left(\frac{r}{\r}\right)^{p2^{n-1}}}{\mu_{2^{n}}^p}\right)
\\ & \le C  \frac{M_p^p(\r,f)}{\left[\mug \left(\left(\frac{r}{\r}\right)^p\right) \right]^p}
 \le C \frac{M_p^p(\r,f)}{\left(\mug\left(\frac{r}{\r}\right)\right)^p}, \quad  \frac{1}{2}\le \frac{r}{\r}<1,
\,\, f\in \H(\D),\,\, 0<p\le 1,
\end{split}
\end{equation}
where in the last inequality we have used  $\mu\in\DD$. Finally, joining \eqref{geq:p>1}, \eqref{geq:im1}, \eqref{geq:im1n}
 and \eqref{geq:mi3} we obtain  \eqref{eq:integralmeans} for $p>1$, and in the case $0<p\le 1$ \eqref{eq:integralmeans} follows from
\eqref{geq:p<1}, \eqref{geq:im1}, \eqref{geq:im1n}
 and \eqref{geq:mi4}.

{\bf{Case $\mathbf{2}$.  $\mathbf{0\le \frac{r}{\r}<\frac{1}{2}}$.}} 
Observe that \eqref{eq:integralmeans} has already been proved for any $\rho>0$ and  $r=\frac{\r}{2}$.
So, 
\begin{equation*}\begin{split}
M_p(r,D^{\mu} f) &\le M_p\left( \frac{\r}{2},D^{\mu} f\right)
\\ &\le 
C(p,\mu) \frac{ M_p\left( \r, f\right)}{\mug\left(\frac{1}{2}\right)}
%\\ &\le C(p,\mu) \frac{ M_p\left( \r, f\right)}{\mug\left(0\right)}
\\ &\le C(p,\mu) \frac{ M_p\left( \r,f\right)}{\mug\left(\frac{r}{\r}\right)}, \quad 0\le \frac{r}{\r}<\frac{1}{2}.
\end{split}\end{equation*}
 This finishes the proof.
\end{proof}

\begin{Prf}{\em{ Theorem~\ref{theorem:L-P-D-hat}.}}
Assume  $\om\in\DD$. 
Without loss of generality we may assume that $f\in A^p_\om$.
So, using that $\mu\in\DD$ and Proposition~\ref{pr:integralmeans}, 
there is $C=C(\om,\mu,p)>0$ such that
\begin{equation}\begin{split}\label{eq:suf1}
\Vert D^{\mu} f \Vert_{A^p_{\mug^p\om}}^p &\le C \int_{\frac{1}{2}}^1 M_p^p(r, D^{\mu}(f)) \om (r)\mug(r)^p dr
\\ & \le 
 C \int_{\frac{1}{2}}^1 M_p^p(\sqrt{r}, f) \frac{\mug(r)^p}{\mug(\sqrt{r})^p} \om (r)\,dr
 \\ & \le C  \int_{\frac{1}{2}}^1 M_p^p(\sqrt{r}, f)  \om (r)\,dr
 \\ & = 2 C \int_{\frac{1}{\sqrt 2}}^1 rM_p^p(r, f) \om (r^2 )dr.
\end{split} \end{equation}
Next, since  $f\in A^p_\om$,
 $$\| f\|^p_{A^p_\om}\ge \int_{r}^{1} M_p^p(s,f)\om(s)\,ds\ge M_p^p(r,f)\omg(r)\to 0\,\quad \text{as $r\to 1^-$}.$$ 

So, two integration by parts and an application of Lemma~\ref{caract. pesos doblantes}~(ii) yield 
\begin{align*}
2 \int_{\frac{1}{\sqrt 2}}^1 r M_p^p(r, f) \om (r^2 )dr
& =  M_p^p \left(\frac{1}{\sqrt{2}}, f\right)\omg\left(\frac{1}{2}\right)+
\int_{\frac{1}{\sqrt 2}}^1 \left[\frac{d}{dr}
M_p^p(r, f)\right] \omg (r^2 )\,dr
\\ & \le  M_p^p \left(\frac{1}{\sqrt{2}}, f\right)\omg\left(\frac{1}{2}\right)+
C\int_{\frac{1}{\sqrt 2}}^1 \left[\frac{d}{dr}
M_p^p(r, f)\right] \omg (r )\,dr, 
\\ &
\le   M_p^p \left(\frac{1}{\sqrt{2}}, f\right)\omg\left(\frac{1}{2}\right)+
C\int_{\frac{1}{ 2}}^1 r M_p^p(r, f) \om (r )dr
\le C\|f\|^p_{A^p_\om},
\end{align*}
which together with \eqref{eq:suf1} implies \eqref{eq:desLP}.

Reciprocally, assume that \eqref{eq:desLP} holds. 
By choosing  $f_n(z)=z^n$,  $n \in \N\cup\{0\}$, in \eqref{eq:desLP} we obtain
$$\int_{\D} \frac{\vert z  \vert^{np}}{\mu_{2n+1}^p}\mug(z)^p \om(z)dA(z)\leq C^p\int_{\D} \vert z\vert^{np} \om(z)dA(z),  \; n \in \N\cup\{0\}. $$
Since $\mu\in\DD$, by Lemma~\ref{caract. pesos doblantes}~(iii) there exists $C=C(\mu)>0$ such that $\mu_{2n+1}\leq C\mug\left( 1-\frac{1}{n+1}\right)$, $n \in \N\cup\{0\}$.
Therefore,
$$\int_{\D} \frac{\vert z  \vert^{np}}{\mug\left( 1-\frac{1}{n+1}\right)^p}\mug(z)^p \om(z)dA(z)\leq C^p\int_{\D} \vert z\vert^{np} \om(z)dA(z), \; n \in \N\cup\{0\}. $$
If $x\geq 1$, we can find $m \in \N$ such that $m\leq x<m+1$. By applying the previous inequality to $n=m+1$, 
$$ \int_0^1 \frac{s^{(m+1)p+1}}{\mug\left( 1-\frac{1}{m+2}\right)^p}\mug(s)^p \om(s)ds\leq C^p\int_0^1 s^{(m+1)p+1} \om(s)ds\leq C^p\int_0^1 s^{xp+1}\om(s)ds, $$
Moreover, bearing in mind the monotonicity of $s^x$ and $\widehat{\mu}(s)$ 
  there exist $C=C(\om, \mu,p)>0$ such that
  \begin{equation*}\begin{split}
\int_0^1 \frac{s^{(m+1)p+1}}{\mug\left( 1-\frac{1}{m+2}\right)^p}\mug(s)^p \om(s)ds
&\geq \int_0^1 \frac{s^{mp+p+1}}{\mug\left( 1-\frac{1}{x}\right)^p}\mug(s)^p \om(s)ds
\\ &\ge C \int_0^1 \frac{s^{mp+1}}{\mug\left( 1-\frac{1}{x}\right)^p}\mug(s)^p \om(s)ds
\\ & \ge C \int_0^1 \frac{s^{xp+1}}{\mug\left( 1-\frac{1}{x}\right)^p}\mug(s)^p \om(s)ds. 
\end{split}\end{equation*}

Therefore, there exists $C=C(\om, \mu,p)> 1$ such that
$$ \int_0^1 \frac{s^{xp}}{\mug\left( 1-\frac{1}{x}\right)^p}\mug(s)^p \om(s)ds\leq C^p\int_0^1 s^{xp}\om(s)ds ,\quad\text{  for all $ x\geq 1$}.$$ 
That is,
\begin{equation}\label{eq:j1}
\int_0^1 s^{xp}\om(s)\left(\left(\frac{\mug(s)}{\mug\left( 1-\frac{1}{x}\right)}\right)^p- C^p\right)ds \leq 0,
\quad\text{  for all $ x\geq 1$}.
\end{equation}
  Take $k_1\ge 1$  such that $\mug\left( 1-\frac{1}{x}\right)<\frac{\mug(0)}{C}$ for $x\geq k_1$. Then, for any $x\ge k_1$ there exists
 $s_x=s_x(x,C, \mu)\in (0,1-\frac{1}{x})$,  the infimum of the points $s\in (0,1-\frac{1}{x})$
such that
$\frac{\mug(s)}{\mug\left( 1-\frac{1}{x}\right)}=C.$ 
By \eqref{eq:j1},
\begin{equation*}\begin{split}
\int_0^{s_x} s^{xp}\om(s)\left(\left(\frac{\mug(s)}{\mug\left( 1-\frac{1}{x}\right)}\right)^p-C^p\right)ds & \leq \int_{s_x}^1 s^{xp}\om(s)\left(C^p-\left(\frac{\mug(s)}{\mug\left( 1-\frac{1}{x}\right)}\right)^p\right)\,ds 
\\ & \leq C^p \omg(s_x),\quad\text{  for all $x \geq k_1$}.
\end{split}\end{equation*}
So,
$$\omg(s_x)\geq C^{-p}\int_0^{s_x} s^{xp}\om(s)\left(\left(\frac{\mug(s)}{\mug\left( 1-\frac{1}{x}\right)}\right)^p-C^p\right)ds. $$
Next, choose  $k_2\geq k_1$ such that $\mug\left( 1-\frac{1}{x}\right)<\frac{\mug(0)}{\left(\frac{3}{2}\right)^{1/p}C}$ if$x\geq k_2$.
 So, for any $x\ge k_2$, 
there exists $r_x=r_x(x,C, \mu)\in (0,1-\frac{1}{x})$, the infimum of the points $r\in (0,1-\frac{1}{x})$
such that $
\frac{\mug(r)}{\mug\left( 1-\frac{1}{x}\right)}= \left(\frac{3}{2}\right)^{1/p}C.$ 
 Then, $r_x < s_x<1-\frac{1}{x}$ and
\begin{equation*}\begin{split}
  \omg(s_x) &\geq C^{-p}\int_0^{s_x} s^{xp}\om(s)\left(\left(\frac{\mug(s)}{\mug\left( 1-\frac{1}{x}\right)}\right)^p-C^p\right)ds
  \\ & \geq \frac{1}{2}  \int_0^{r_x} s^{xp}\om(s)ds, \quad\text{ for any $x\geq k_2$.}
\end{split}\end{equation*}
By  Fubini's theorem, 

\begin{equation*}\begin{split}
  2\omg(s_x)&\geq  \int_0^{r_x} s^{xp}\om(s)ds=\int_0^{r_x} \om(s)\left(\int_0^s px t^{xp-1}dt\right) ds
  \\ & = \int_0^{r_x} px s^{xp-1} (\omg(s)-\omg(r_x))ds
  \\ & \geq\int_0^{r} px s^{xp-1}\omg(s)ds-\omg(r_x)\int_0^{r_x} px s^{xp-1}ds
  \\ & \geq\omg(r)r^{px}-\omg(r_x)r_x^{px}, \; 0<r<r_x<1.
\end{split}\end{equation*}
Then, for any $x\ge k_2$ and $0<r<r_x<1$,
$$ \omg(r)r^{px}\leq 2\omg(s_x)+\omg(r_x)r_x^{px}\leq (2+r_x^{px})\omg(r_x)\leq 3\omg(r_x).$$

It is clear that $r_x>\frac{1}{2}$ if $k_2$ is large enough. In this case, take $r=2r_x-1$ in the previous inequality, that is
%$r_x=\frac{1+r}{2}$.
\begin{equation}\label{eq:th11}
 \omg(r)\leq 3 r^{-px}\omg\left(\frac{1+r}{2}\right), \quad 0<r=2r_x-1<r_x<1, \quad x\ge k_2.
 \end{equation}
 Since $\mu \in \Dd$, by Lemma~\ref{caract. D check}~(ii)  there exist $C_2=C_2(\mu)>0$ and $\b=\b(\mu)>0$ such that
$$\left(\frac{3}{2}\right)^{1/p}C= \frac{\mug(r_x)}{\mug\left( 1-\frac{1}{x}\right)}\geq C_2(x(1-r_x))^{\b}.$$ 
So, 
$$ x\leq \frac{\left(\frac{(\frac{3}{2})^{\frac{1}{p}}C}{C_2}\right)^{\frac{1}{\b}}}{1-r_x}=\frac{2\left(\frac{(\frac{3}{2})^{\frac{1}{p}}C}{C_2}\right)^{\frac{1}{\b}}}{1-r}.$$
Then, for $r\geq r_0= 1-\frac{2}{k_2}\left(\frac{(\frac{3}{2})^{\frac{1}{p}}C}{C_2}\right)^{\frac{1}{\b}},$ 
$$ r^{px}\geq \left(\inf\limits_{r_0\leq r<1} r^{\frac{1}{1-r}}\right)^{2\left(\frac{(\frac{3}{2})^{\frac{1}{p}}C}{C_2}\right)^{\frac{1}{\b}}p}=C_3=C_3(\omega,\mu,p)>0,$$
which together with \eqref{eq:th11}
 yields $\omg(r)\leq 3 C_3 \omg(\frac{1+r}{2})$, for $r_0\leq r <1$. Therefore $\om \in \DD.$

\end{Prf}

\section{Proof of Theorem~\ref{th:L-P-D}.}
We begin this section proving a technical result on $L^p$-integrability of power series with nonnegative coefficients.
We use ideas from the proofs of \cite[Theorem~6]{MatelPav} and \cite[Proposition~9]{PelRathg}.

\begin{proposition}
\label{integral real suma}
Let $0<p<\infty$, $\eta \in \DDD$  and $k \in \N\setminus\{1\}$  such that \eqref{D chek y k} holds for $\eta$. 
Let be $f(r)=\sum\limits_{j=0}^{\infty}a_j r^j$ 
where $a_j\geq 0$ for all $j \in \N \cup \{0\}$. If $t_0=\sum\limits_{j=0}^{k-1}a_j$ and
 $t_n=\sum\limits_{j=k^{n}}^{k^{n+1}-1} a_j$, $n \in \N$.  Then, there exist positive constants $C_1=C_1(p, \eta)>0$ and 
 $C_2=C_2(p, \eta)>0$
  such that
$$ C_1 \sum\limits_{n=0}^{\infty}\eta_{k^n}t_n^p \leq \int_0^1 f(s)^p \eta(s)\,ds\leq C_2 \sum\limits_{n=0}^{\infty}\eta_{k^n}t_n^p.$$

\end{proposition} 
\begin{proof}
First, we show the lower estimate
\begin{align*}
\int_0^1 f(s)^p \eta(s)\,ds &\geq \sum\limits_{n=0}^{\infty} \int_{1-\frac{1}{k^{n+1}}}^{1-\frac{1}{k^{n+2}}}\left(\sum_{m=0}^{k-1}a_ms^m+\sum\limits_{j=1}^{\infty}\sum\limits_{m=k^j}^{k^{j+1}-1}a_m s^m\right)^p\eta(s)ds
\\ &\geq \sum\limits_{n=0}^{\infty} \int_{1-\frac{1}{k^{n+1}}}^{1-\frac{1}{k^{n+2}}} s^{pk^{n+1}} t_n^p \eta(s)ds
\\& \ge C(p,\eta) \sum\limits_{n=0}^{\infty} t_n^p \int_{1-\frac{1}{k^{n+1}}}^{1-\frac{1}{k^{n+2}}} \eta(s)ds \\
&=C(p,\eta) \sum\limits_{n=0}^{\infty} t_n^p \left(\etag\left(1-\frac{1}{k^{n+1}}\right)- \etag\left(1-\frac{1}{k^{n+2}}\right)\right).
\end{align*}
Since \eqref{D chek y k} holds for $k$ and $\eta$, there exists a constant $C=C(\eta)>1$ such that  
$\etag\left(1-\frac{1}{k^{n+1}}\right)\geq C \etag\left(1-\frac{1}{k^{n+2}}\right)$.
This together with  Lemma \ref{caract. pesos doblantes}(iii) yields
$$\int_0^1 f(s)^p \eta(s)\,ds\ge C(p,\eta)\sum\limits_{n=0}^{\infty} t_n^p\etag\left(1-\frac{1}{k^{n+2}}\right)
\ge C(p,\eta) \sum\limits_{n=0}^{\infty}\eta_{k^n}t_n^p.$$

In order to show the upper bound, observe that
\begin{equation}\label{eq:rs1}
f(s)\leq a_0+ \sum\limits_{n=0}^{\infty} s^{k^n} t_n.
\end{equation}
If $0<p\leq 1$, by
Lemma \ref{caract. pesos doblantes}~(v), there is $C=C(p,\eta)>0$ such that
\begin{equation*}\begin{split}
 \int_0^1 f(s)^p \eta(s)ds &\le 
 a_0^p\etag(0)+\int_0^1 \sum\limits_{n=0}^{\infty} s^{k^n p} t_n^p \eta (s)ds
\\ &=  a_0^p\etag(0)+ \sum\limits_{n=0}^{\infty} t_n^p \eta_{k^{n}p}
\\ & 
  \le C\left( t_0^p \eta_{k^0}+  \sum\limits_{n=0}^{\infty} t_n^p \eta_{k^n}
  \right) \le C \sum\limits_{n=0}^{\infty} t_n^p \eta_{k^n}.
 \end{split}\end{equation*}
If $1<p<\infty$, take $\gamma$ such that $0<\frac{\gamma p}{p'}<1$. Then, by \eqref{eq:rs1}, Hölder's inequality and Lemma \ref{suma}, we obtain
\begin{align*}
f(s)\leq a_0+ \left(\sum\limits_{n=0}^{\infty} \frac{s^{k^n}}{\eta_{k^n}^{\gamma}} \right)^{\frac{1}{p'}}\left(\sum\limits_{n=0}^{\infty} s^{k^n}t_n^p \eta_{k^n}^{\frac{\gamma p}{p'} } \right)^{\frac{1}{p}}\lesssim a_0+ \frac{1}{\etag(s)^{\frac{\gamma}{p'}}}\left(\sum\limits_{n=0}^{\infty} s^{k^n}t_n^p \eta_{k^n}^{\frac{\gamma p}{p'} } \right)^{\frac{1}{p}}
\end{align*}
which yields
\begin{equation}\begin{split}\label{eq:rs2}
\int_0^1 f(s)^p \eta(s)\,ds &\lesssim a_0^p\etag(0)+ \sum\limits_{n=0}^{\infty} t_n^p \eta_{k^n}^{\frac{\gamma p}{p'}} \int_0^1 s^{k^n} 
 \frac{\eta (s)}{\etag(s)^{\frac{\gamma p}{p'}}} ds
 \\ & \lesssim t_0^p \eta_{k^0}+ \sum\limits_{n=0}^{\infty} t_n^p \eta_{k^n}^{\frac{\gamma p}{p'}} \int_0^1 s^{k^n}  \frac{\eta (s)}{\etag(s)^{\frac{\gamma p}{p'}}} ds.
 \end{split}\end{equation}
Next, let us prove that $$ \eta_{k^n}^{\frac{\gamma p}{p'}} \int_0^1 s^{k^n}  \frac{\eta (s)}{\etag(s)^{\frac{\gamma p}{p'}}} ds\lesssim \eta_{k^n},$$
which together with \eqref{eq:rs2}  finishes the proof. Indeed, by Lemma~\ref{caract. pesos doblantes}~(iii)
$$  \eta_{k^n}^{\frac{\gamma p}{p'}} \int_0^{1-\frac{1}{k^n}} s^{k^n}  \frac{\eta (s)}{\etag(s)^{\frac{\gamma p}{p'}}} ds\leq \frac{ \eta_{k^n}^{\frac{\gamma p}{p'}}}{\etag\left(1-\frac{1}{k^n}\right)^{\frac{\gamma p}{p'}}} \int_0^{1} s^{k^n} \eta (s) ds \asymp \eta_{k^n},$$ %where the last equivalence is due to Lemma \ref{caract. pesos doblantes} (iii).
 Morever,  an integration and another application of Lemma~\ref{caract. pesos doblantes}~(iii) imply
 $$ \eta_{k^n}^{\frac{\gamma p}{p'}} \int_{1-\frac{1}{k^n}}^1 s^{k^n}  \frac{\eta (s)}{\etag(s)^{\frac{\gamma p}{p'}}} ds \leq \eta_{k^n}^{\frac{\gamma p}{p'}} \int_{1-\frac{1}{k^n}}^1  \frac{\eta (s)}{\etag(s)^{\frac{\gamma p}{p'}}} ds\asymp
  \eta_{k^n}^{\frac{\gamma p}{p'}} \etag\left(1-\frac{1}{k^n}\right)^{1-\frac{\gamma p}{p'}} \lesssim \eta_{k^n}. $$
This finishes the proof.
\end{proof}

The right choice of the norm used is in many cases a key to a good understanding of how a
concrete operator acts in a given space. Here, the following decomposition result provides 
an effective tool for the study  of the fractional derivative $D^\mu$.

\begin{proposition}
\label{desc. norma bloques momentos y Vn}
Let $0<p<\infty$,  $\eta \in \DDD$ and $k=k(\eta)>1$, $k\in \N$ such that \eqref{D chek y k} holds for $\eta$ and $k$. If
$\{ V_{n,k}\}_{n=0}^{\infty}$ is a sequence of polynomials considered in Proposition~\ref{pr:cesaro}, then there are
constants $C_1=C_1(p, \eta)>0$ and $C_2=C_2(p, \eta)>0$ such that
$$ C_1\sum\limits_{n=0}^{\infty} \eta_{k^n}\Vert V_{n,k}\ast f\Vert_{H^p}^p\leq \nm{f}_{A_{\eta}^p}^p\leq C_2 \sum\limits_{n=0}^{\infty} \eta_{k^n}\Vert V_{n,k}\ast f\Vert_{H^p}^p, \quad f\in\H(\D).$$
\end{proposition}
\begin{proof}
 By \eqref{propervn2} and \cite[Lemma~3.1]{MatPavStu84}, there is $C>0$ such that
$$ \nm{f_r}_{H^p}\geq C^{-1}\Vert V_{n,k}\ast f_r\Vert_{H^p}\ge C^{-1} r^{k^{n+1}}\Vert V_{n,k}\ast f\Vert_{H^p}$$
for any $0\le r<1$ and $n\in\N$. So,
$$ \nm{f_r}_{H^p}\geq C^{-1} \sup\limits_{n \in \N}r^{k^{n+1}}\Vert V_{n,k}\ast f\Vert_{H^p},$$
which implies
\begin{equation}\begin{split}\label{eq:norma1}
\nm{f}_{A^p_{\eta}}^p &\asymp \int_0^1 \nm{f_r}_{H^p}^p \eta(r)dr \geq C \sum\limits_{n=0}^{\infty} \int_{1-\frac{1}{k^{n+1}}}^{1-\frac{1}{k^{n+2}}} \left(\sup\limits_{j \in \N}r^{k^{j+1}}\Vert V_{j,k}\ast f\Vert_{H^p}\right)^p \eta (r)dr
\\ & \geq C \sum\limits_{n=0}^{\infty} \Vert V_{n,k}\ast f\Vert_{H^p}^p \int_{1-\frac{1}{k^{n+1}}}^{1-\frac{1}{k^{n+2}}} r^{k^{n+1}p} \eta(r)dr 
\\ & \geq  C \sum\limits_{n=0}^{\infty} \Vert V_{n,k}\ast f\Vert_{H^p}^p \left(\etag\left(1-\frac{1}{k^{n+1}}\right)- \etag\left(1-\frac{1}{k^{n+2}}\right)\right), \quad f\in \H(\D).
\end{split}\end{equation}
Since \eqref{D chek y k} holds for $k$ and $\eta$, there exists a constant $C=C(\eta)>1$ such that 
$\etag\left(1-\frac{1}{k^{n+1}}\right)\geq C \etag\left(1-\frac{1}{k^{n+2}}\right)$, which together with 
\eqref{eq:norma1} and Lemma \ref{caract. pesos doblantes}(iii), yields
\begin{equation}\begin{split}\label{eq:norma2}
 \nm{f}_{A^p_{\eta}}^p \gtrsim \sum\limits_{n=0}^{\infty} \Vert V_{n,k}\ast f\Vert_{H^p}^p \etag\left(1-\frac{1}{k^{n+2}}\right) \asymp \sum\limits_{n=0}^{\infty} \Vert V_{n,k}\ast f\Vert_{H^p}^p \eta_{k^n} , \quad f\in \H(\D). 
 \end{split}\end{equation}
In order to show the reverse inequality, we distinguish two cases according to the range of $p$.
If $0<p\leq 1$, by using \eqref{propervn1}
 and \cite[Lemma~3.1]{MatPavStu84} we obtain
\begin{align*}
\nm{f_r}_{H^p}^p &=  \left\|\sum\limits_{n=0}^{\infty} V_{n,k}\ast f_r\right\|_{H^p}^p
\\ & \lesssim \sum\limits_{n=0}^{\infty} \left\| V_{n,k}\ast f_r\right\|_{H^p}^p 
\\ & \lesssim \|V_{0,k}\ast f\|_{H^p}^p +\sum\limits_{n=1}^{\infty} \left\| V_{n,k}\ast f\right\|_{H^p}^p  r^{k^{n-1}p},
\quad f\in \H(\D),
\end{align*}
and therefore by  Lemma \ref{caract. pesos doblantes}~(v),
\begin{equation}\begin{split}\label{eq:norma3}
\nm{f}_{A^p_{\eta}}^p & \lesssim \|V_{0,k}\ast f\|_{H^p}^p \etag(0)+ \sum\limits_{n=1}^{\infty} \left\| V_{n,k}\ast f\right\|_{H^p}^p \int_0^1 r^{k^{n-1}p}\eta (r)dr
\\ & \lesssim  \sum\limits_{n=0}^{\infty} \left\| V_{n,k}\ast f\right\|_{H^p}^p \eta_{k^n}, \quad f\in \H(\D),\,\,0<p\le 1.
\end{split}\end{equation}
On the other hand,  if $1<p<\infty$, by \eqref{propervn1} and \cite[Lemma~3.1]{MatPavStu84}, we obtain
\begin{align*}
\nm{f_r}_{H^p}^p & = \left\|\sum\limits_{n=0}^{\infty} V_{n,k}\ast f_r\right\|_{H^p}^p \leq \left(\sum\limits_{n=0}^{\infty} \left\| V_{n,k}\ast f_r\right\|_{H^p}\right)^p \\
&\lesssim \left(\|V_{0,k}\ast f\|_{H^p} +\sum\limits_{n=1}^{\infty} \left\| V_{n,k}\ast f\right\|_{H^p}  r^{k^{n-1}}\right)^p
\\&
\lesssim \|V_{0,k}\ast f\|_{H^p}^p +\left(\sum\limits_{n=1}^{\infty} \left\| V_{n,k}\ast f\right\|_{H^p}  r^{k^{n-1}}\right)^p,
\quad f\in \H(\D).
\end{align*}
The above chain of inequalities together with Proposition~\ref{integral real suma} yields
\begin{equation}\begin{split}\label{eq:norma4}
\nm{f}_{A^p_{\eta}}^p & \lesssim \|V_{0,k}\ast f\|_{H^p}^p \etag(0)+ \int_0^1 \left( \sum\limits_{n=1}^{\infty} \left\| V_{n,k}\ast f\right\|_{H^p} r^{k^{n-1}}\right)^p \eta (r)dr
\\ & \lesssim \sum\limits_{n=0}^{\infty} \left\| V_{n,k}\ast f\right\|_{H^p}^p \eta_{k^n},
\quad f\in \H(\D),\,\, 1<p<\infty.
\end{split}\end{equation}
Consequently, joining \eqref{eq:norma2}, \eqref{eq:norma3} and \eqref{eq:norma4}, the proof is finished.
\end{proof}

With Proposition~\ref{desc. norma bloques momentos y Vn} in hand, we are able to prove 
that the space of analytic functions $D^p_{\om,\mug}=\{f \in\H(\D): \int_\D|D^{\mu}(f)(z)|^p\widehat{\mu}(z)^p\om(z)
\,dA(z)<\infty\}$ is continuuosly embedded into $A_{\om}^p$ when $\om\in\DDD$ and $\mu\in\DD$.
This result together with Theorem~\ref{theorem:L-P-D-hat} proves that 
\eqref{Eq:L-P-D} holds when $\omega\in\DDD$.

\begin{theorem}\label{th:inversesuf}
Let $\om \in \DDD$,  $0<p<\infty$  and  $\mu \in \DD$. Then there exists $C=C(\om,\mu,p)>0$ such that 
$$ \nm{f}_{A_{\om}^p}\leq C\nm{D^{\mu}(f)}_{A^p_{\om \mug^p}}, \quad f \in \H(\D).$$
\end{theorem}
\begin{proof}
By Lemma \ref{caract. D check}(iii) there exists $k=k(\om)>1$, $k \in \N$ such that \eqref{D chek y k} holds for $k$ and $\omega$. 
Next, Lemma \ref{prop pesos}~(ii) ensures that $\om \mug^p\in \DDD$ and
Lemma \ref{prop pesos} (i) implies that  $\om \mug^p$ satisfies \eqref{D chek y k} with the same $k$ as $\omega$ does.
Therefore, we can apply Proposition \ref{desc. norma bloques momentos y Vn} to the weights $\om,\om\mug^p \in \DDD$   and the choosen $k$.
That is, there are positive constants $C_j(\om,p)>0$, $j=1,2$ such that
\begin{equation}
\label{desc. en bloques para omega}
 C_1(\om,p) \sum\limits_{n=0}^{\infty} \om_{k^n}\Vert V_{n,k}\ast f\Vert_{H^p}^p 
\le \nm{f}_{A_{\om}^p}^p\le C_2(\om,p) \sum\limits_{n=0}^{\infty} \om_{k^n}\Vert V_{n,k}\ast f\Vert_{H^p}^p, \quad f \in \H(\D),
\end{equation}
and $C_j(\om,\mu,p)>0$, $j=3,4$ such that
\begin{equation}
\label{desc. en bloques para omegamug}
 C_3(\om,\mu,p) \sum\limits_{n=0}^{\infty} (\om\mug^p)_{k^n}\Vert V_{n,k}\ast D^{\mu}f\Vert_{H^p}^p
\le \nm{D^{\mu}f}_{A_{\om\mug^p}^p}^p\le
C_4(\om,\mu,p) \sum\limits_{n=0}^{\infty} (\om\mug^p)_{k^n}\Vert V_{n,k}\ast D^{\mu}f\Vert_{H^p}^p,
\end{equation}
for all $f \in \H(\D).$

Observe that for each $n\in\N$, $(V_{n,k}\ast f)(z)=\sum\limits_{j=k^{n-1}}^{k^{n+1}-1} \widehat{f}(j)\widehat{V_{n,k}}(j)z^j$ and 
$$(V_{n,k}\ast D^{\mu}f)(z)=\sum\limits_{j=k^{n-1}}^{k^{n+1}-1} \frac{\widehat{f}(j)}{\mu_{2j+1}}\widehat{V_{n,k}}(j)z^j.$$ So, applying \cite[Lemma~9(i)]{PelRatproj} to $g=V_{n,k}\ast D^{\mu} (f)$, $h=V_{n,k}\ast  f$ and $S_{k^{n-1}, k^{n+1}-1} h= V_{n,k} \ast f$, there exists a constant $C=C(p)>0$ such that 
\begin{equation}
\label{sacar coef. norma bloques}
\nm{V_{n,k}\ast f}_{H^p}\leq C \mu_{k^{n-1}}\nm{V_{n,k}\ast D^{\mu}(f)}_{H^p}, \quad f \in \H(\D).
\end{equation}
A similar argument, shows that
\begin{equation}
\label{sacar coef. norma bloques0}
\nm{V_{0,k}\ast f}_{H^p}\leq C \mu_{0}\nm{V_{0,k}\ast D^{\mu}(f)}_{H^p}, \quad f \in \H(\D).
\end{equation}
Moreover, by Lemma \ref{caract. pesos doblantes}  and  \eqref{D chek y k} (for $k$ and $\omega$),
there is $C=(\om,\mu,p)>0$ such that
\begin{equation}\begin{split} \label{des. momentos}
\om_{k^n}\mu_{k^{n-1}}^p 
 & \le C\om_{k^n}\mu_{k^{n+1}}^p 
\\ & \le C\omg\left(1-\frac{1}{k^{n}}\right)\mug\left(1-\frac{1}{k^{n+1}}\right)^p
\\ & \le C\mug\left(1-\frac{1}{k^{n+1}}\right)^p \int_{1-\frac{1}{k^n}}^{1-\frac{1}{k^{n+1}}} \om(s)\,ds
\\ & \le C \int_{1-\frac{1}{k^n}}^{1-\frac{1}{k^{n+1}}} \om(s)\mug(s)^p\,  ds 
\\ & \le C (\om \mug^p)_{k^n}, \quad \text{for all $n \in \N$.}
\end{split}\end{equation}

Then, by joining \eqref{desc. en bloques para omega}, \eqref{sacar coef. norma bloques}, \eqref{sacar coef. norma bloques0}, \eqref{des. momentos} and \eqref{desc. en bloques para omegamug},
\begin{align*}
\nm{f}_{A_{\om}^p}^p & \le C\sum\limits_{n=0}^{\infty} \om_{k^n}\Vert V_{n,k}\ast f\Vert_{H^p}^p 
\\ & \le C\left[\om_1\mu^p_{0}\nm{V_{0,k}\ast D^{\mu}(f)}_{H^p}^p+
\sum\limits_{n=1}^{\infty} \om_{k^n}\mu_{k^{n-1}}^p\Vert V_{n,k}\ast D^{\mu}(f)\Vert_{H^p}^p \right]\\
& \le C \sum\limits_{n=0}^{\infty} (\om\mug^p)_{k^{n}}\Vert V_{n,k}\ast D^{\mu}(f)\Vert_{H^p}^p 
\le C \nm{D^{\mu}(f)}_{A^p_{\om \mug^p}}^p, \quad f \in \H(\D).
\end{align*}
This finishes the proof.
\end{proof}

In order to complete a proof of Theorem~\ref{th:L-P-D} we have to show that 
$\om\in\DDD$ is a necessary condition so that \eqref{Eq:L-P-D} holds. This implication will follow from Theorem~\ref{theorem:L-P-D-hat}
and the next result. 

\begin{theorem}\label{th:inverseM}
Let $\om$ be a radial weight, 
 $0<p<\infty$ and $\mu\in\DDD$. If there exists $C=C(\om,\mu,p)>0$ such that
\begin{equation}
\label{inverse ineq Lp}
\Vert f \Vert_{A^p_{\om}}^p \leq C \int_{\D}\vert D^{\mu} (f)(z) \vert^p \om(z)\mug(z)^p dA(z), \quad f \in \H(\D),    
\end{equation}
then $\om \in \M$.
\end{theorem}

\begin{proof}
By choosing $f_n(z)=z^n\in \H(\D)$, $n\in\N$ in (\ref{inverse ineq Lp}), we deduce 
\begin{align*}
\int_{\D} \vert z\vert^{np} \om(z)\,dA(z)
& \leq C^p\int_{\D} \frac{\vert z  \vert^{np}}{\mu_{2n+1}^p}\mug(z)^p \om(z)\,dA(z)
\\ & \le C^p \int_{\D} \frac{\vert z  \vert^{np}}{\mug\left( 1-\frac{1}{n+1}\right)^p}\mug(z)^p \om(z)dA(z), \; n \in \N. 
\end{align*}
Now, if $x\geq 1$, we can find $m \in \N$ such that $m\leq x<m+1$. Then, 
 bearing in mind the monotonicity of $s^x$ and $\widehat{\mu}(s)$ 
\begin{align*}
\int_0^1 s^{xp+1}\om(s)ds &\leq \int_0^1 s^{mp+1} \om(s)ds
\\ & \leq C^p  \int_0^1 \frac{s^{mp+1}}{\mug\left( 1-\frac{1}{m+1}\right)^p}\mug(s)^p \om(s)ds
\\ & \leq C^p  \int_0^1 \frac{s^{(m+1)p+1}}{\mug\left( 1-\frac{1}{m+1}\right)^p}\mug(s)^p \om(s)ds
\\ & \le C^p\int_0^1 \frac{s^{xp+1}}{\mug\left( 1-\frac{1}{x+1}\right)^p}\mug(s)^p \om(s)ds,\quad x\ge 1.
\end{align*}
That is, there exists $C=C(p, \om, \mu)> 1$ such that
$$\int_0^1 s^{xp+1}\om(s)\left(\frac{1}{C^p}-\left(\frac{\mug(s)}{\mug\left( 1-\frac{1}{x+1}\right)}\right)^p\right)ds \leq 0.$$

Since $\mug\left( 1-\frac{1}{x+1}\right)<C\mug(0)$ for all $x\geq 1$, there exists $s_x=s_x(x,C, \mu)$,  the supremum of the points $s\in (1-\frac{1}{x},1)$ such that
$\frac{\mug(s)}{\mug\left( 1-\frac{1}{x+1}\right)}=\frac{1}{C}.$ 
Then,
\begin{equation*}\begin{split}
&\int_{s_x}^1 s^{xp+1}\om(s)\left(\frac{1}{C^p}-\left(\frac{\mug(s)}{\mug\left( 1-\frac{1}{x+1}\right)}\right)^p\right)ds
\\ & \leq \int_0^{s_x} s^{xp+1}\om(s)\left(\left(\frac{\mug(s)}{\mug\left( 1-\frac{1}{x+1}\right)}\right)^p-\frac{1}{C^p}\right)\,ds,\quad x\ge 1. 
\end{split}\end{equation*}

There also exists $r_x=r_x(x,C, \mu) \in (s_x,1)$ 
the supremum of the points $r\in (1-\frac{1}{x},1)$
such that
$\frac{\mug(r)}{\mug\left( 1-\frac{1}{x+1}\right)}=\frac{1}{C}\left(\frac{1}{2}\right)^{\frac{1}{p}}.$ 
So,
\begin{align*}
\int_{s_x}^1 s^{xp+1}\om(s)\left(\frac{1}{C^p}-\left(\frac{\mug(s)}{\mug\left( 1-\frac{1}{x+1}\right)}\right)^p\right)ds&\geq \frac{1}{2C^p}\int_{r_x}^1 s^{xp+1}\om(s)ds\\
&=\frac{1}{2C^p}\om_{xp+1}-\frac{1}{2C^p}\int_0^{r_x} s^{xp+1}\om(s)\,ds,\quad x\ge 1.
\end{align*}
Therefore,
\begin{equation}\begin{split}\label{eq:in1}
\om_{xp+1}&\leq 2C^p \int_0^{s_x} s^{xp+1}\om(s)\left(\left(\frac{\mug(s)}{\mug\left( 1-\frac{1}{x+1}\right)}\right)^p-\frac{1}{C^p}\right)ds+\int_0^{r_x} s^{xp+1}\om(s)ds 
\\ &\leq 2C^p \int_0^{r_x} s^{xp+1}\om(s)\left(\frac{\mug(s)}{\mug\left( 1-\frac{1}{x+1}\right)}\right)^pds+\int_0^{r_x} s^{xp+1}\om(s)ds,
\quad x\ge 1.
\end{split}\end{equation}
Next, by Lemma \ref{caract. pesos doblantes}(ii),  
there exist $C_1=C_1(\mu)>1$ and $\a=\a(\mu)>0$ such that
\begin{equation}\begin{split}\label{eq:in2}
\mug(s)\leq C_1 \mug(r_x)\left(\frac{1-s}{1-r_x}\right)^{\a}
=C_1\left(\frac{1}{2}\right)^{\frac{1}{p}}\frac{1}{C}\mug\left(1-\frac{1}{x+1}\right)\left(\frac{1-s}{1-r_x}\right)^{\a},\quad 0<s\le r_x,
\end{split}\end{equation}
where  in the last identity we have used the definition of $r_x$.
Consequently, putting together \eqref{eq:in1} and \eqref{eq:in2}
\begin{equation}\begin{split}\label{eq:th101}
 \om_{xp+1}
 \leq 
\frac{C_1^p}{(1-r_x)^{\a p}} (\om_{[\a p]})_{xp+1},\quad x\ge 1.
 \end{split}\end{equation}
On the other hand, by Lemma \ref{caract. D check}~(ii) there exist $C_2=C_2(\mu)>0$ and $\b=\b(\mu)>0$ such that
$$\left(\frac{1}{2}\right)^{\frac{1}{p}}\frac{1}{C}=\frac{\mug(r_x)}{\mug\left(1-\frac{1}{x+1}\right)}\leq C_2(2x(1-r_x))^{\b},$$
so there is $C_3=C_3(p,\om,\mu)>0$ such that  
$\frac{1}{1-r_x}\le C_3 x$. This together  with \eqref{eq:th101} implies that there is $C=C(\om,\mu,p)>0$ such that
\begin{equation*}
\om_{xp+1}\leq C x^{\a p}(\om_{[\a p]})_{xp+1},\quad x\ge 1.
\end{equation*}
That is,
\begin{equation*}
\om_{y}\leq C \left(\frac{y-1}{p} \right)^{\a p}(\om_{[\a p]})_{y}\le
C\left(\frac{1}{p} \right)^{\a p} y^{\a p}(\om_{[\a p]})_{y} ,\quad y\ge p+1,
\end{equation*}
which together with Lemma~\ref{caract. M} implies that $\om\in\M$.  This finishes the proof.
\end{proof}

Finally, we are ready to prove Theorem~\ref{th:L-P-D}.

\begin{Prf}{\em{ Theorem~\ref{th:L-P-D}.}}
If $\om\in\DDD$, putting together Theorem~\ref{theorem:L-P-D-hat} and Theorem~\ref{th:inversesuf}, we get
\eqref{Eq:L-P-D}. Reciprocally, if \eqref{Eq:L-P-D} holds, $\om\in\DD$  by Theorem~\ref{theorem:L-P-D-hat}  and  $\om\in \M$
by
Theorem~\ref{th:inverseM}. Then, it follows from \cite[Theorem~3]{PR19} that 
$\om\in \DD\cap\M=\DDD$. 

This finishes the proof.
\end{Prf}
\medskip

We would like to point out that it would be interesting to obtain some progress about Littlewood-Paley inequalities for fractional derivatives on
Bergman spaces $A^p_\omega$ induced by a non-radial weight $\omega$. For instance, to
know whether or  not \eqref{Eq:L-P-D} ($\mu\in\DDD$) remains true for   Bekoll\'e-Bonami weights.

\end{document}